\documentclass[11pt,reqno]{amsart}
\usepackage{lineno}
\usepackage{graphicx}
\usepackage{aliascnt}
\usepackage{amsmath}
\usepackage{amsthm}
\usepackage{latexsym,bm}
\usepackage{amssymb}
\usepackage{indentfirst}
\usepackage{enumerate}
\usepackage[bookmarks=false,pdfstartview=FitH, pdfborder={0 0 0}, colorlinks=true,citecolor=blue, linkcolor=blue]{hyperref}
\usepackage[numbers, sort&compress]{natbib}

\newtheorem{thm}{Theorem}[section]

\newaliascnt{lemma}{thm}
\aliascntresetthe{lemma}

\newtheorem{lem}[lemma]{Lemma}

\newtheorem{rem}[thm]{Remark}
\newtheorem{exm}{Example}

\numberwithin{equation}{section}

\newcommand{\Id}{\operatorname{Id}}

\newcommand{\CC}{\mathbb{C}}
\newcommand{\HH}{\mathbb{H}}
\newcommand{\R}{\mathbb{R}}
\newcommand{\Z}{\mathbb{Z}}
\newcommand{\kahler}{K\"{a}hler }
\newcommand{\diam}{\operatorname{diam}}

\renewcommand{\sec}{\operatorname{sec}}
\newcommand{\Ric}{\operatorname{Ric}}
\newcommand{\rr}{\sqrt}
\newcommand{\coker}{\operatorname{coker}}

\DeclareMathOperator{\ind}{ind}
\DeclareMathOperator{\tr}{Tr}
\DeclareMathOperator{\Td}{td}
\DeclareMathOperator{\cok}{coker}

\newcommand {\be}{\begin{equation}}
\newcommand {\ee}{\end{equation}}

\begin{document}

\title{Almost Nonnegative Ricci curvature and new vanishing theorems for genera}

\author{Xiaoyang Chen}
\address[Chen]{School of Mathematical Sciences, Institute for Advanced Study, Tongji University, Shanghai, China.}
\email{xychen100@tongji.edu.cn}
\author{Jian Ge}
\address[Ge]{School of Mathematical Sciences, Laboratory of Mathematics and Complex Systems, Beijing Normal University, Beijing 100875, P. R. China.}
\email{jge@bnu.edu.cn}
\author{Fei Han}
\address[Han]{Department of Mathematics, National University of Singapore}
\email{mathanf@nus.edu.sg}
\maketitle

\begin{abstract}
We derive several vanishing theorems for genera under almost nonnegative Ricci curvature and infinite fundamental group, which includes Todd genus, $\widehat{A}$-genus, elliptic genera and Witten genus. A vanishing theorem of Euler characteristic number for almost nonnegatively curved Alexandrov spaces is also proved.
\end{abstract}

\tableofcontents


\section{Introduction}

A classical theorem of Lichnerowicz asserts that a compact spin manifold carrying a Riemannian metric of positive scalar curvature has vanishing $\widehat{A}$-genus. In a different direction, Lott conjectured that a compact spin manifold  with \emph{almost nonnegative sectional curvature} has vanishing $\widehat{A}$-genus \cite{Lot2000}. In this paper we derive several vanishing theorems for various genera under almost nonnegative Ricci curvature condition. We start from the Todd genus.

\begin{thm}\label{nice}
Let $M$ be a compact complex manifold with infinite fundamental group. If $M$ admits a sequence of \kahler metrics with almost nonnegative Ricci curvature, then the Todd genus of $M$ vanishes.
\end{thm}

If a sequence of Riemannian metrics $\{g_i\}_{n\in \mathbb N}$ on a smooth manifold $M$ satisfies
\[
\Ric (g_i) \geq -1\quad{\text{and}}\quad \diam(g_i) \leq \frac{1}{i},
\]
for all $i\in\mathbb N$, then we say that $\{g_i\}_{n\in \mathbb N}$ have almost nonnegative Ricci curvature. Here $\Ric (g_i)$ and $\diam (g_i)$ stand for the Ricci curvature and diameter of $g_i$, respectively. If $g$ has nonnegative Ricci curvature, it is clear that $g_i =\epsilon_i g, \epsilon_i=\frac{1}{i^2 \diam (g)^2}$ have almost nonnegative Ricci curvature.

Given a compact complex manifold $M^n$ of complex dimension $n$, its Todd genus (or holomorphic Euler number) is defined to be
\[
\sum_{p=0}^n (-1)^p \dim H^{0, p} (M, \mathbb{C}),
\]
where $H^{0, p} (M, \mathbb{C})$ is the Dolbeault cohomology group of $M$. We emphasize that it is necessary to assume that $M$ has infinite fundamental group in \autoref{nice}. To see this, note that the Todd genus of complex projective space is $1$.

\autoref{nice} can be viewed as a complex analogue of a theorem of Fukaya-Yamaguchi \cite{FY1992}, which says that a compact Riemannian manifold with \emph{almost nonnegative sectional curvature} and infinite fundamental group has vanishing topological Euler number. However, a compact Riemannian manifold with almost nonnegative Ricci curvature and infinite fundamental group may have nonzero topological Euler number, see the example constructed by Anderson in \cite{And1992}.

\begin{exm}
Let $X^{n+1}$ be a smooth abelian variety of complex dimension $(n+1)\geq 3$ embedded in some complex projective space $\mathbb{CP}^N$. Let $M^n$ be the intersection of $X$ with some generic $\mathbb{CP}^{N-1}$. Then $M^n$ is a smooth hypersurface of $X$ with positive holomorphic normal bundle $L$. By Lefschetz hyperplane theorem, the fundamental group of $M$ is an infinite abelian group. Moreover, the Todd genus of $M$ is nonzero. In fact, by Hirzebruch-Riemann-Roch theorem \cite{Bal2006, Hir1995}, the Todd genus of $M$ is computed by $\int_{M} \Td(T M)$, where $\Td(T M)$ is the Todd class of the tangent bundle of $M$. Let $d$ be the first Chern class of $L$. As $X$ is an abelian variety, then $\Td(T X)=1$. Moreover, by properties of Todd class \cite{Hir1995}, we have
\[
\Td(T M) \Td(L)=\Td(T X)|_{M},
\]
\[
\Td(L)=\frac{d}{1-e^{-d}},\ d=c_1(L).
\]
Then
\[
\Td(T M)=\frac{1-e^{-d}}{d}.
\]
As $L$ is a positive line bundle, we see that
\[
\int_{M} \Td(T M)=(-1)^n \int_{M} \frac{d^n} {(n+1)!} \neq 0.
\]
Hence the Todd genus of $M$ is nonzero. It follows that $M \times \mathbb{CP}^k$ has nonzero Todd genus for any $k \geq 1$. Therefore by \autoref{nice}, $M \times \mathbb{CP}^k$ does not admit a sequence of \kahler metrics with almost nonnegative Ricci curvature. For $k \geq 2$, it seems that all previous known obstructions to almost nonnegative Ricci curvature do not apply to $M \times \mathbb{CP}^k$.
\end{exm}

Riemannian manifolds with almost nonnegative Ricci curvature have been studied extensively \cite{Ber1988, Fan2002, Gro1981, KL2019, KPT2010, KW2011, Zha2007}.
We briefly recall some previously known results (here $m=\dim  M$).
\begin{itemize}
\item The fundamental group of $M$ has a nilpotent subgroup of finite index \cite{KW2011}.
\item The first Betti number of $M$ is bounded above by $m$ \cite{Gro1981, Gal1983} with equality being achieved if and only if $M$ is diffeomorphic to a torus \cite{CC1997, Col1997}.
\item If $M$ is spin and $m$ is divisible by $4$, then its $\widehat{A}$-genus is bounded from above by $2^{\frac{m}{2}}$ \cite{Gro1981, Gal1983}.
\end{itemize}

Compact \kahler manifolds with almost nonnegative Ricci curvature have been studied in \cite{Fan2002, Zha2007}. Under the additional assumptions that $M$ has bounded sectional curvature and other things, it was shown in \cite{Zha2007} that $M$ has finite fundamental group. Moreover, $M^n$ is diffeomorphic to a complex manifold $X$ such that the universal covering of $X$ has a decomposition: $ \widetilde{X}= X_1 \times \cdots \times X_s$, where $X_i$ is a Calabi-Yau manifold, or a hyper\kahler manifold, or $X_i$  satisfies $H^{p,0}(X_i ,\mathbb{C}) = 0, p > 0.$ If $M$ has infinite fundamental group, under the stronger assumption that $M$ has almost nonnegative bisectional curvature and bounded sectional curvature, it was shown in \cite{Fan2002} that there is a holomorphic fibration $ M \rightarrow J(M)$, where $J(M)$ is the Jacobian of $M$, a complex torus of dimension $\frac{1}{2}b_1(M)$ and $b_1$ is the first Betti number of $M$.

Compared with the works in \cite{Fan2002, Zha2007}, we do \emph{not} assume boundedness of sectional curvature in \autoref{nice}.

Our method can also be used to prove the following vanishing result for $\widehat{A}$-genus.
\begin{thm} \label{spin}
Let $M$ be a $4m$-dimensional compact spin manifold with infinite fundamental group. If $M$ admits a sequence of Riemannian metrics with almost nonnegative Ricci curvature, then its $\widehat{A}$-genus vanishes.
\end{thm}

Lott \cite{Lot2000} conjectured that a compact spin manifold $M$ with \emph{almost nonnegative sectional curvature} has vanishing $\widehat{A}$-genus. \autoref{spin} weakens the assumption from almost nonnegative sectional curvature to almost nonnegative Ricci curvature under the additional assumption that $M$ has infinite fundamental group. We emphasize that this additional assumption is necessary. In fact, the $K3$ surface is a simply connected spin manifold which admits a Ricci flat metric. However, its $\widehat{A}$-genus is nonzero.

As an application of \autoref{spin}, we give the following example, where all  previous known obstructions to almost nonnegative Ricci curvature do not apply.
\begin{exm}
Let $B^8$ be a Bott manifold, which is spin, simply connected and $\widehat A(B^8)=1$ (c.f. Section 4 in \cite{RS1995}). Let $M^8=(\mathbb{T}^2 \times S^6) \sharp B^8$. Then $M^8$ is a spin manifold with infinite fundamental group and
\[
\widehat{A}(M^8)=\widehat{A}(\mathbb{T}^2 \times S^6)+\widehat{A}(B^8)=1.
\]
By \autoref{spin}, $M^8$ can not admit a sequence of Riemannian metrics with almost nonnegative Ricci curvature.
\end{exm}

The elliptic genera and Witten genus are $q$-deformations of the $\widehat{A}$-genus and $\widehat{L}$-genus. They are partition functions of two-dimensional quantum field theories and formally the indices of Dirac operator in free loop spaces.  If the formal Chern roots of the complexified tangent bundle $TM\otimes \CC$ of a $4m$-dimensional compact oriented manifold $M$ are $\{\pm 2\pi \rr{-1} x_j, 1\leq j
\leq 2m\}$, then the Hirzebruch $\widehat{A}$-genus and $\widehat{L}$-genus are the characteristic number of $M$ defined by
\[
\widehat{A}(M)=\left\langle \prod_{j=1}^{2m} \frac{\pi\rr{-1}x_j}{\sinh(\pi\rr{-1}x_j)}, [M] \right\rangle,
\]
\[
\widehat{L}(M)=\left\langle \prod_{j=1}^{2m} \frac{2\pi \rr{-1}x_j}{\tanh(\pi\rr{-1}x_j)}, [M] \right\rangle,
\]
while the elliptic genera $Ell_1(M), Ell_2(M)$ and the Witten genus $W(M)$ are defined
by
\begin{equation}
Ell_1(M)=\left\langle \left( \prod_{j=1}^{2m}x_{j}\frac{\theta ^{\prime }(0,\tau
)}{\theta (x_{j},\tau )}\frac{\theta_1(x_j,\tau
)}{\theta_1(0,\tau )}\right)\!, [M]\right\rangle=\widehat{L}(M)+O(q),
\end{equation}
\begin{equation}
Ell_2(M)=\left\langle \left( \prod_{j=1}^{2m}x_{j}\frac{\theta ^{\prime }(0,\tau
)}{\theta (x_{j},\tau )}\frac{\theta_2(x_j,\tau
)}{\theta_2(0,\tau )}\right)\!, [M]\right\rangle=\widehat{A}(M)+O(q^{1/2}),
\end{equation}
\begin{equation}
W(M)=\left\langle \left( \prod_{j=1}^{2m}x_{j}\frac{\theta ^{\prime }(0,\tau
)}{\theta (x_{j},\tau )}\right)\!, [M]\right\rangle=\widehat{A}(M)+O(q^{1/2}),
\end{equation}
where $\theta, \theta_1, \theta_2$ are the classical Jacobi theta functions.
The theory of elliptic genera and Witten genus lead to the profound theory of elliptic cohomology \cite{Hop}. See \autoref{EWgenus} for some basic facts on the elliptic genera and Witten genus.

We can also prove a vanishing theorem for elliptic genera and Witten genus. In the following, $R_{g_i}$ and $V_i$ stand for the Riemannian curvature tensor  and volume of $g_i$ respectively.
\begin{thm} \label{string}
Let $M$ be a $4m$-dimensional compact manifold with infinite fundamental group. If for some constants $p>2m$ and $\Lambda>0$, there exists a sequence of Riemannian metrics $\{g_i\}_{i\in\mathbb N}$ on $M$ such that the following holds for all $i$,
\[
\Ric (g_i) \geq -1,\quad \diam (g_i) \leq \frac{1}{i},
\]
\[
\frac{1}{V_i}\int_{M} |R_{g_i}|^p dV_i \leq \Lambda,
\]
then
\begin{enumerate}[(i)]
\item $M$ has vanishing elliptic genera if $M$ is spin. In particular, $M$ has vanishing signature.
\item $M$ has vanishing Witten genus if $M$ is further string (i.e. spin and the spin class $\frac{p_1(M)}{2}=0$).
\end{enumerate}
\end{thm}

The integral curvature assumption in \autoref{string} is technical. We do not know if it is indeed necessary.

\begin{exm} (i) Let $N^8=(\mathbb{T}^2 \times S^6) \sharp \HH P^2$, where $\HH P^2$ is the quarternionic projective plane. Then $N^8$ is a spin manifold with infinite fundamental group and the signature
\[
\sigma(N^8)=\sigma(\mathbb{T}^2 \times S^6)+\sigma(\HH P^2)=1\neq 0.
\]
So $N^{8}$ can not admit a sequence of Riemannian metrics with almost nonnegative Ricci curvature as well as the additional integral curvature bound in \autoref{string}. Note that the $\widehat{A}$-genus of $N^8$ is 0.

(ii) Let $W^{24}$ be a $24$-dimensional simply connected string manifold with nonzero Witten genus and vanishing $\widehat{A}$-genus. Let $M^{24}=(\mathbb{T}^2 \times S^{22}) \sharp  W^{24}$. Then the Witten genus
\[
\phi_W(M^{24})=\phi_W(\mathbb{T}^2 \times S^{22})+\phi_W(W^{24})=\phi_W(W^{24})\neq 0.
\]
So $M^{24}$  can not admit a sequence of Riemannian metrics with almost nonnegative Ricci curvature and the additional integral curvature bound in \autoref{string}.
Note that $M^{24}$ admits a Riemannian metric with positive scalar curvature.
\end{exm}

Some special cases of \autoref{nice}, \autoref{spin} and \autoref{string} are easy to prove: If $M$ has infinite fundamental group and admits a Riemannian metric
with nonnegative Ricci curvature, then by the Splitting Theorem of \cite{CG}, some finite cover of $M$ is diffeomorphic to a product manifold $\mathbb{T}^k \times N$ for some $k \geq 1$. It follows that the genera of $M$ appearing in \autoref{nice}, \autoref{spin} and \autoref{string} are zero.

On the other hand, based on a fibration theorem proved in \cite{KPT2010}, a vanishing theorem on signature is proved in \cite{BLO} for Riemannian manifolds with \emph{almost nonnegative sectional curvature} and infinite fundamental group. However, no such fibration theorem exists under almost nonnegative Ricci curvature condition, by the counterexample constructed in \cite{And1992}. In order to prove our vanishing results on genera, we employ an analytic method based on the Bochner technique. Via the Atiyah-Singer index theorem, our strategy is to show that the indices of certain Dirac operators are zero. Using the work of B\'{e}rard \cite{Ber1988}, we can derive a bound of the indices of those Dirac operators. In the presence of infinite fundamental group, our crucial contribution is to show that the indices are in fact zero. In this step the multiplicity property of genera under a finite covering is used in an essential  way. See \autoref{sec:2} for the details of proof.

In \cite{CH2021}, we also obtained several vanishing theorems for genera under different curvature assumptions. We emphasize that the methods used in the proof in \cite{CH2021} and this paper are different.

{\bf Acknowledgements} Xiaoyang Chen is partially supported by National Natural Science Foundation of China No.11701427. He thanks Professor Botong Wang for helpful discussions. Fei Han is partially supported by the grant AcRF A-8000451-00-00 from National University of Singapore. Jian Ge is partially supported by National Key R\&D Program of China grant 2020YFA0712800, NSFC 11731001.

\section{Proof of \autoref{nice}, \autoref{spin} and \autoref{string}}\label{sec:2}
In this section we prove \autoref{nice}, \autoref{spin} and \autoref{string}.

First we prove \autoref{nice}. Let $M$ be a compact complex manifold of complex dimension $n$ with infinite fundamental group. We are going to show that the Todd genus of $M$ vanishes if $M$ admits a sequence of \kahler metrics $g_i$ satisfying
\[
\Ric (g_i) \geq -1,\quad \diam(g_i) \leq \frac{1}{i}.
\]
Then the fundamental group of $M$ contains a nilpotent subgroup of finite index \cite{KW2011}. Let $M_1$ be a finite cover of $M$ such that $\pi_1 (M_1)$ is nilpotent. Then we have the following algebraic lemma where the assumption that $M$ has infinite fundamental group will be used.

\begin{lem} \label{re}
There exists a sequence of normal subgroups of finite index
\[
\pi_1(M_1)=G_1 \supset G_2 \supset G_3 \supset \cdots \supset G_j\supset \cdots
\]
such that $\lim_{j \rightarrow +\infty} [G_1: G_j]  =\infty$.
\end{lem}
\begin{proof}
If $\pi_1(M_1)$ is free abelian, i.e. $\pi_1(M_1)\simeq \mathbb{Z}^k$ for some $k\geq 1$, we can simply take $G_j=(2^{j-1} \mathbb{Z})^k, j \geq 1$. In general, we will apply a theorem of Hirsch \cite{H}. Since $\pi_1(M_1)$ is a finitely generated nilpotent group, by Hirsch's theorem \cite{H}, it is residually finite, i.e. the intersection of all normal subgroups of finite index is the unit element. Let $G_j$ be the intersection of all normal subgroups of $\pi_1(M_1)$ with index $\leq j$, $j=1,2,\ldots.$ By a theorem of Hall \cite{Ha}, in every finitely generated group, the number of subgroups of given finite index is finite. Then we see that $G_j$ is a normal subgroup of $\pi_1(M_1)$ with finite index. From the definition of $G_j$, we have
\[
\pi_1(M_1)=G_1 \supset G_2 \supset G_3 \supset \cdots \supset G_j\supset \cdots,
\]
\[
\bigcap_j G_j=\{{1}\}.
\]
Moreover, we have
\[
\lim_{ j \rightarrow +\infty}  [G_1: G_j]  =\infty.
\]
Otherwise, there exists some constant $C>0$ such that $[G_1: G_j] \leq C$ for all $j$. As $[G_1: G_j]$ is a monotone increasing sequence taking values in $\mathbb{Z}$, there must exist some $j_0$ such that $[G_1: G_j]=[G_1: G_{j_0}]$ for all $j \geq j_0$. It turns out that $G_j=G_{j_0}$ for each $j \geq j_0$. Hence $\bigcap_j G_j=G_{j_0}.$ Since $\bigcap_j G_j=\{1\}$, then $G_{j_0}=\{1\}.$ As $[G_1: G_{j_0}] < \infty$, it follows that $\pi_1(M_1)=G_1$ is a finite group, which contradicts that $M$ has infinite fundamental group.
\end{proof}

The Dolbeault cohomology group of a compact  \kahler manifold $M$ satisfy $\dim H^{0,p}(M, \mathbb{C})=\dim H^{p,0}(M, \mathbb{C})$ for any $p$.
Let
\[
D_i={\partial} + {\partial}^{*}: \oplus_k \wedge^{2k,0}(M) \rightarrow \oplus_k \wedge^{2k+1,0}(M)
\]
be the first order elliptic operators on $M$ determined by $g_i$. Then the Todd genus of $M$ is equal to the index of $D_i$ for any $i$ by Hodge theory. Denote $\widehat{g_i}^j, \widehat{D_i}^j $ be the pulling backs of $g_i, D_i$ to $M_j$, where $M_j$ is the finite covering of $M_1$ with $\pi_1(M_j)=G_j$. Then we can compare the diameter of $\widehat{g_i^j}$ with that of $\widehat{g_i^1}$:
\begin{lem}[Ivanov \cite{Iva2010}] \label{iv}
\[
\diam (\widehat{g_i^j}) \leq  \ [G_1: G_j] \ \diam (\widehat{g_i^1}).
\]
\end{lem}
\begin{proof}
We include the proof given by Sergei Ivanov \cite{Iva2010} on Mathoverflow for completeness. Denote $[G_1: G_j]=N$. Suppose the contrary, let $\gamma^j: [0, \ell]\to M^j$ be a geodesic realizes the diameter of $(M^j, \widehat{g_i^j})$ with $\ell>N\diam (\widehat{g_i^1})$. Let $\gamma$ be the image of $\gamma^j$ in $M_1$ under the canonical projection. We broke $\gamma$ into $N$ shorter geodesics of equal length $\ell/N$. Namely
\[
\gamma:=a_1a_2\cdots a_N.
\]
Since $\ell/N>\diam (\widehat{g_i^1})$, there must exists a strictly shorter geodesic $b_i$ connecting the two ends of $a_i$ for each $i\in \{1, \cdots, N\}$. Set $\bar p=\gamma^j(0)$, $p=\gamma(0)$ and $\bar q=\gamma^j(\ell)$, $q=\gamma(\ell)$.  Let $H$ be the subgroup of $\pi_1(M_1, p)$ whose elements(loops) lift to closed loops in $M_j$. Clearly $[\pi_1(M_1, p), H]=N$. Now consider the following $(N+1)$ loops based at $p$:
\[
s_0=e, s_1=b_1a^{-1}_1, s_2=b_1b_2(a_1a_2)^{-1},\cdots, s_N=b_1\cdots b_N(a_1\cdots a_N)^{-1}.
\]
There must exist two element $s_k$ and $s_{r}$ with $k<r$ such that $(s_k)^{-1}s_{r}\in H$. Now the path $(s_k)^{-1}s_{r}\gamma$ still lifts $p$ and $q$ to $\bar{p}$ and $\bar{q}$, yet it is strictly shorter than $\gamma^j$, a contradiction.
\end{proof}

Now we derive a bound of the index of $\widehat{D_i}^j$.
\begin{lem} \label{bound}
There exists a positive constant $C(n)$ depending only on the dimension $n$ such that for each $i, j$, we have
\[
|\ind (\widehat{D_i}^j)| \leq C(n).
\]
\end{lem}

\begin{rem}
It is essential for us that the constant $C(n)$ is independent of $i, j$.
\end{rem}

\begin{proof}
We will prove \autoref{bound} using the method in \cite{Ber1988}. It is based on the following lemmas.

\begin{lem}\label{peter}
For a smooth section $s$ of an Euclidean vector bundle $E$ over a Riemannian manifold $(M,g)$, let
\[
\|s\|_{\infty} = \sup \{{|s|_x: x \in M}\},
\]
\[
\|s\|_2^2=\int_{M} \|s\|_x^2 dV_g,
\]
\[
L(s)=\frac{V_g \|s\|_{\infty}^2 }{\|s\|_2^2},
\]
where $V_g$ is the volume of $(M,g)$. Given a finite dimensional subspace $F$ of $C^{\infty} (E)$, we have
\[
\dim F \leq l \sup\{{L(s): s \in F-0}\},
\]
where $l$ is the rank of $E$.
\end{lem}

We refer the readers to \cite{Ber1988} for the proof of \autoref{peter}.

\begin{lem}\label{sb1}
Let $(M, g)$ be a closed $m$-dimensional smooth Riemannian manifold such that for some constant $b>0$,
\[
r_{\min}(g) (\diam(g))^2 \geq-(m-1)b^2.
\]
If $f\in W^{1,2}(M)$ is a nonnegative continuous functions such that $f \Delta f \geq -h f^2$ (here $\Delta$ is a negative operator)
in the sense of distribution for some nonnegative continuous function $h$ then
\[
\max_{x \in M}|f|^2(x) \leq C(m, p, R, \Lambda) \frac{\int_M f^2 dV}{V(g)},
\]
where $C(m, p, R, \Lambda)$ is some constant depending only on $m, p, R=\frac{\diam(g)}{b C(b)}$ and
\[
\Lambda=\frac{\int_M h^p dV}{V(g)}, p > \frac{m}{2}.
\]
\end{lem}

We refer the readers to see \cite{CH2021} or  \cite{Ber1988} for the proof of Lemma \ref{sb1}.
Several notations in the above lemma need to be clarified here:

$(1) \ r_{\min}(g)=\inf\{{\Ric(g)(u,u): u \in TM, g(u,u)=1}\}$

(2) $V(g)$ is the volume of $(M, g)$.

(3)\ $C(b)$ is the unique positive root of the equation $ x \int_0^{b}(cht + x sht)^{m-1} dt=\int_0^{\pi} \sin^{m-1}t dt.$

The following explicit expression of the constant $C(m, p, R, \Lambda)$ derived in \cite{CH2021} is important for us.
Let $v=\frac{m}{2}$ if $m>2$ and $ 1 < v < p $ be arbitrary for  $m=2$. Let $\mu$ be the conjugate of $v$ such that
\[
\frac{1}{v} + \frac{1}{\mu}=1.
\]
Then
\begin{equation} \label{ex}
C(m, p, R, \Lambda)=\mu^{2K_1 \frac{p(\mu-1)}{\mu (p-1)-p}} B^{2 K_2}
\end{equation}

\begin{equation}
B=C(m,p) \Lambda^{\frac{1}{2} \frac{\mu-1}{\mu(p-1)-p}} R^{\frac{p(\mu-1)}{\mu (p-1)-p}} +2
\end{equation}

\begin{equation}
K_1= \sum i {\mu}^{-i},  \ K_2=\sum {\mu}^{-i},
\end{equation}
where $C(m,p)$ is some positive constant depending only on $m, p$.

Now we apply \autoref{peter} and \autoref{sb1} to prove \autoref{bound}, namely,
\[
|\ind (\widehat{D_i}^j)| \leq C(n), \ n=\dim_{\mathbb{C}} M.
\]

Let $E_j=\oplus_k \wedge^{2k, 0}(M_j), F_{i, j}=\ker (\widehat{D_i}^j).$ Then for any $s \in \ker (\widehat{D_i}^j)$, applying the Bochner formula to $s$, we get
\[
\frac{1}{2}\Delta |s|^2 = |\nabla s|^2  + \langle \mathcal{F} s, s \rangle.
\]
As $g_i$ is a \kahler metric for each $i$, it is crucial here that the curvature term $\langle \mathcal{F} s, s \rangle$ is controlled by the Ricci curvature, see for example Page 82 in \cite{Bal2006}.
By assumption, we have
\[
\Ric(g_i) \geq  -1, \ \diam (g_i) \leq \frac{1}{i}.
\]
Then $\Ric(\widehat{g_i}^j) \geq  -1$ and there exists some positive constant $C(n)$ depending only on $n$ such that
\[
\frac{1}{2}\Delta |s|^2 \geq |\nabla s|^2 - C(n) |s|^2.
\]

By Kato's inequality  \cite{Ber1988}, we have
\[|\nabla s| \geq |\nabla |s||.
\]

Then we get
\[
|s| \Delta |s| \geq - C(n) |s|^2.
\]

Let $G_j=\pi_1(M_j), G_0=\pi_1(M)$. By \autoref{iv}, we get
\[\diam (\widehat{g_i}^j) \leq  [G_1: G_j] \diam(\widehat{g_i}^1)
\leq [G_1: G_j] \ [G_0: G_1] \diam(g_i).
\]
As $\diam (g_i) \leq \frac{1}{i}$, we get
\[ \diam (\widehat{g_i}^j) \leq \frac{[G_1: G_j] \ [G_0: G_1]}{i}.
\]
\par
Applying \autoref{sb1} to $|s|$ for any $p>n$ (for example take $p=n+1$), we get
\[
|s|^2_{\infty}=:\max_{x \in M}|s|^2(x) \leq   C(n,i,j) \frac{\int_{M} |s|^2 dV_i}{V(g_i)},
\]
where $C(n,i,j)$ is some positive constant depending on $n, i, j$ (which can be described explicitly via \autoref{ex}).

Applying \autoref{peter} to $F_{i, j}$, we get
\[
\dim (\ker (\widehat{D_i}^j)) \leq C(n, i, j).
\]
As $\widehat{D_i}^j$ is self-adjoint, we can also show that
\[
\dim (\coker (\widehat{D_i}^j)) \leq C'(n, i, j),
\]
where $C'(n,i,j)$ is some positive constant depending on $n, i, j$ (which can be described explicitly via \autoref{ex}).
\par
Then
\[
|\ind (\widehat{D_i}^j)| \leq \max(C(n, i, j), C'(n, i, j)).
\]

\par
For a fixed $j$, we have that $\ind (\widehat{D_i}^j)$ is independent of $i$. Since $\Ric(\widehat{g_i}^j) \geq  -1$ and $\lim_{i \rightarrow \infty} \diam (\widehat{g_i}^j)=0$ for a fixed $j$, then we get
\[
|\ind (\widehat{D_i}^j)|=\lim_{i \rightarrow \infty} |\ind (\widehat{D_i}^j)| \leq C(n),
\]
where $C(n)$ is some positive constant depending only on $n$. It is essential for us that the constant $C(n)$ is independent of $i, j$.
\end{proof}

Now we are going to finish the proof of \autoref{nice}. As Todd genus is multiplicative under a finite covering \cite{Hir1995}, we have
\[
|\ind (\widehat{D_i}^1)|\ [G_1: G_j] =|\ind (\widehat{D_i}^j)|.
\]
By \autoref{re} and \autoref{bound}, we get
\[
|\ind (\widehat{D_i}^j)| \leq C(n), \ \lim_{j \rightarrow \infty} [G_1: G_j]=\infty.
\]
Hence for each $i$, we must have
\[
|\ind (\widehat{D_i}^1)| =0.
\]
It follows that $M_1$ has vanishing Todd genus, so does $M$.

The proof of \autoref{spin} is similar to that of \autoref{nice}. Here we look at the Atiyah-Singer spin Dirac operator $D$. The curvature term of the Bochner formula applied to the kernel of $D$ is bounded from below by the scalar curvature. Then we can apply \autoref{peter} to get a bound of $\hat{A}$-genus. This in fact was already proved in
\cite{Gro1981, Gal1983}. In the presence of infinite fundamental group, then our arguments given above can be used to show that the $\hat{A}$-genus is in fact zero.

Now we deal with the proof of \autoref{string}. By \autoref{mo}, $Ell_2(M)$ is determined by $ \ind (D\otimes B_k(T_\CC M)), 0 \leq k \leq \frac{m}{2}$, where $B_k(T_\CC M)$ involves linear combinations of tensor product of $TM$ at most to power ${[\frac{m}{2}]}$. The kernel of $D\otimes B_j(T_\CC M)$ is the space of twisted harmonic spinors. Then the curvature term of the Bochner formula applied to twisted harmonic spinors is bounded from below by $-C(m) |R_{g_i}|$, where $C(m)$ is some positive constant depending only on $m$. The rest part of proof is almost identical to the proof of \autoref{nice}. As we do not assume a pointwise bound of the Riemannian curvature tensor, in order to get a bound of the dimension of the space of twisted harmonic spinors, the integral curvature assumption $\frac{1}{V_i}\int_{M} |R_{g_i}|^p dV_i \leq \Lambda$ will be used. In the presence of infinite fundamental group, then our arguments can be used to show that $Ell_2(M)=0$. By \autoref{mod1},
\[
Ell_1(M, -1/\tau)={(2\tau)}^{2m}Ell_2(M, \tau)=0
\]
In particular, we get that the first term in the $q$-expansion of $Ell_1(M)$: the signature $\sigma(M)=0$.  This proves Part (i), concerning the vanishing of elliptic genera in \autoref{string}.

Concerning the Witten genus, when $M$ is string, the Witten genus $W(M)$ is determined by finitely many (depending only on the dimension of $M$) $D\otimes W_j(T_\CC M)$ by \autoref{mod0}, the rest of the proof is identical to the one for elliptic genera.


\section{Vanishing theorems on Euler number}

The Euler number of a manifold is not a genus. However, certain vanishing results about Euler number have been obtained under \emph{almost nonnegative sectional curvature} assumption. Recall that a smooth Riemannian manifold $M$ has almost nonnegative sectional curvature if it admits a sequence of Riemannian metrics $\{g_i\}_{i\in \mathbb N}$ such that
\[
\sec (g_i) \geq -1,\quad
\diam(g_i) \leq \frac{1}{i},
\]
where $\sec(g_i)$ denotes the sectional curvatures of $g_i$.

Fukaya-Yamaguchi proved that a closed manifold with almost nonnegative sectional curvature and infinite fundamental group has vanishing Euler number. Their proof is based on a fibration theorem \cite{FY1992} characterizing these manifolds. We are going to give a new proof of this result based on the idea developed in previous proofs. In fact our proof works for a larger class of spaces: almost nonnegatively curved Alexandrov spaces. Alexandrov spaces are generalizations of Riemannian manifolds with  lower sectional  curvature bound. They could have topological singularities, for example the spherical cone over $\mathbb{CP}^n$ with Fubini-Study metric is an Alexandrov space. For background of Alexandrov spaces, cf \cite{BGP1992, BBI2001}. Note that due to the possible singularities on Alexandrov spaces, there is no fibration theorem available for almost nonnegatively curved 
  Alexandrov spaces. Yet, we still have:

\begin{thm}\label{qr}
Let $X$ be an $m$-dimensional almost nonnegatively curved Alexandrov space with infinite fundamental group, then the Euler number of $X$ vanishes.
\end{thm}
\begin{proof}
By \cite{Ya}, the fundamental group of $X$ contains a nilpotent subgroup of finite index. By passing to a finite cover, we can assume that $\pi_1 (X)$ is nilpotent. By \autoref{re}, there is a sequence of subgroups $G_j$ of $\pi_1(X)$ with finite index such that $\lim_{j \rightarrow +\infty} [G_1: G_j]  =\infty$. Then  $\widetilde{X}/ G_j$ is a finite cover of $X$, where $\widetilde{X}$ is the universal covering of $X$. Clearly, $\widetilde{X}/ G_j$ is almost nonnegatively curved by pulling back the metrics. By the extension of Gromov's Betti number estimate to Alexandrov spaces (see the main Theorem of \cite{LS1994}), we get that for any $j$,
\[
\sum_{p=0}^m b_p(\widetilde{X}/ G_j)\le C(m).
\]
This in particular implies that the absolute value of Euler number of $\widetilde{X}/ G_j$ is bounded above by some constant depending only on $m$.
As $\lim_{j \rightarrow +\infty} [G_1: G_j]  =\infty$, we see that the Euler number of $X$ is zero by the multiplicity property of Euler number under a finite covering.
\end{proof}

The same idea can also be used to prove a vanishing theorem about $L^2$ Betti numbers under almost nonnegative sectional curvature condition:
\begin{thm}\label{L2qr}
Let $M$ be a compact manifold (or more generally a compact Alexandrov space) with almost nonnegative sectional curvature and infinite fundamental group, then all the $L^2$ Betti numbers of its universal covering $\widetilde{M}$ vanish.
\end{thm}

We refer the reader to \cite{Lue2002} for the precise definition of $L^2$-Betti numbers. We only recall the following two facts on $L^2$-Betti numbers. See pages 37,  51 and 476 in \cite{Lue2002} for details.
\begin{lem}[\cite{Lue1994}]\label{lem2}
Let $M$ be a finite connected CW-complex with fundamental group $G$. Suppose there exists a sequence of normal subgroups of finite index
\[
G=G_1 \supset G_2 \supset G_3 \supset \cdots \supset G_j \supset \cdots
\]
such that  $\bigcap_j G_j=\{{1}\}$, then the $p$-th $L^2$ Betti number of $\widetilde{M}$ is equal to
\[
\lim_{j \rightarrow +\infty} \frac{b_p(\widetilde{M} / G_j)}{[G_1: G_j]}.
\]
\end{lem}

\begin{lem}\label{lem1}
The Euler number of $M$ is equal to the alternative sum of $L^2$ Betti numbers of $\widetilde{M}$.
\end{lem}

\begin{proof}[Proof of \autoref{qr}]
By passing to a finite cover, we can assume that $\pi_1 (M)$ is nilpotent. By \autoref{re}, there is a sequence of subgroups $G_j$ of $\pi_1(M)$  with finite index such that $\lim_{j \rightarrow +\infty} [G_1: G_j]  =\infty$.
Since the Betti numbers $b_p(\widetilde{M} / G_j)$ is bounded above by some constant depending only on $\dim M$, then we have
\[
\lim_{j \rightarrow +\infty} \frac{b_p(\widetilde{M} / G_j)}{[G_1: G_j]} =0.
\]
By \autoref{lem2}, we see that the $p$-th $L^2$ Betti number of $\widetilde{M}$ vanishes.
\end{proof}

\begin{rem}
The above proof also works for Alexandrov spaces. In fact by \cite{MY2019}, any compact Alexandrov space $X$ admits a good cover, it follows that $X$ is homotopic to the nerve of a good cover of $X$. Then L\"{u}ck's \autoref{lem2} still applies.
\end{rem}

\begin{rem}
It follows from \autoref{L2qr} and \autoref{lem1} that the Euler number of $M$ vanishes. Therefore this provides a second proof of \autoref{qr} without using fibration theorem for almost nonnegatively curved manifolds.
\end{rem}

In general, it is not true that the Euler number of a compact manifold with almost nonnegative Ricci curvature and infinite fundamental group is zero. See the example in \cite{And1992}. However, our method can also be used to prove a vanishing theorem of Euler number under an additional assumption on curvature operators.
\begin{thm} \label{euler}
Let $M$ be a compact manifold with infinite fundamental group. If for some constant $C>0$, there exists a sequence of Riemannian metrics $(g_i)_{i\in \mathbb N}$ on $M$ with almost nonnegative Ricci curvature such that the curvature operator of $g_i$ is bounded below by $-C \Id$, then the Euler number of $M$ vanishes.
\end{thm}
The proof of \autoref{euler} is similar to that of \autoref{nice}. Here we look at the Dirac operator $d + d^*$. The kernel of $d + d^*$ is the space of harmonic forms. As the curvature operator of $g_i$ is bounded below by $-C\Id$,  then the curvature term of the Bochner formula applied to harmonic forms is bounded from below by some negative constant depending only on $C$ and $\dim M$. The rest part of proof is identical to the proof of \autoref{nice}.

\appendix
\section{Some basic facts on elliptic genus and  Witten genus }\label{EWgenus}

In this section, we recap some basic facts about elliptic genus and Witten genus. Let $M$ be a $4m$-dimensional compact oriented smooth manifold and $\{\pm 2\pi \sqrt{-1}x_{j},1\leq j\leq 2m\}$ denote the formal Chern roots of $T_{%
\mathbb{C}}M $, the complexification of the tangent vector bundle $TM$ of $M$.

Let
\[
\widehat A(M)=\prod_{j=1}^{2m}\frac{\pi\sqrt{-1}x_j}{\sinh(\pi\sqrt{-1}x_j)}, \ \ \ \widehat L(M)=\prod_{j=1}^{2m}\frac{2\pi\sqrt{-1}x_j}{\tanh(\pi\sqrt{-1}x_j)}
\]
be the $\widehat A$-class and $\widehat L$-class of $M$ respectively.

The elliptic genera $Ell_1(M), Ell_2(M)$ and Witten genus $W(M)$ are defined
by (c.f. \cite{Liu95cmp})
\begin{equation}
Ell_1(M)=\left\langle \left( \prod_{j=1}^{2m}x_{j}\frac{\theta ^{\prime }(0,\tau
)}{\theta (x_{j},\tau )}\frac{\theta_1(x_j,\tau
)}{\theta_1(0,\tau )}\right)\!, [M]\right\rangle=\widehat{L}(M)+O(q),
\end{equation}
\begin{equation}
Ell_2(M)=\left\langle \left( \prod_{j=1}^{2m}x_{j}\frac{\theta ^{\prime }(0,\tau
)}{\theta (x_{j},\tau )}\frac{\theta_2(x_j,\tau
)}{\theta_2(0,\tau )}\right)\!, [M]\right\rangle=\widehat{A}(M)+O(q^{1/2}),
\end{equation}
\begin{equation}
W(M)=\left\langle \left( \prod_{j=1}^{2m}x_{j}\frac{\theta ^{\prime }(0,\tau
)}{\theta (x_{j},\tau )}\right)\!, [M]\right\rangle=\widehat{A}(M)+O(q^{1/2}),
\end{equation}
where $\theta, \theta_1, \theta_2$ are the classical Jacobi theta functions:
\be \theta(v,\tau)=2q^{1/8}\sin(\pi v)\prod_{j=1}^\infty[(1-q^j)(1-e^{2\pi \sqrt{-1}v}q^j)(1-e^{-2\pi
\sqrt{-1}v}q^j)], \ee
\be \theta_1(v,\tau)=2q^{1/8}\cos(\pi v)\prod_{j=1}^\infty[(1-q^j)(1+e^{2\pi \sqrt{-1}v}q^j)(1+e^{-2\pi
\sqrt{-1}v}q^j)], \ee
 \be \theta_2(v,\tau)=\prod_{j=1}^\infty[(1-q^j)(1-e^{2\pi \sqrt{-1}v}q^{j-1/2})(1-e^{-2\pi
\sqrt{-1}v}q^{j-1/2})], \ee
$q=e^{2\pi \sqrt{-1}\tau}, \tau\in \mathbb{H}$.

Let $E$ be a complex vector bundle. For any complex number $t$, let
\[
\Lambda_t(E)={\mathbb C}|_M+tE+t^2\Lambda^2(E)+\cdots ,\quad S_t(E)={\mathbb C}|_M+tE+t^2S^2(E)+\cdots
\]
denote respectively the total exterior and symmetric powers  of $E$, which live in
$K(M)[[t]]$. Denote $\widetilde{E}=E-\CC^{\mathrm{rk} E}$ in $K(M)$. Elliptic genus and Witten genus can be expressed as (c.f. \cite{CHZ11, Liu1995})
\begin{equation}
Ell_1(M)=\left\langle \widehat{L}(TM)\mathrm{ch}\left( \Theta \left( T_{\mathbb{C%
}}M\right)\otimes  \Theta_1 \left( T_{\mathbb{C%
}}M\right) \right) ,[M]\right\rangle ,
\end{equation}%

\begin{equation}
Ell_2(M)=\left\langle \widehat{A}(TM)\mathrm{ch}\left( \Theta \left( T_{\mathbb{C%
}}M\right)\otimes  \Theta_2 \left( T_{\mathbb{C%
}}M\right) \right) ,[M]\right\rangle ,
\end{equation}%

\begin{equation}
W(M)=\left\langle \widehat{A}(TM)\mathrm{ch}\left( \Theta \left( T_{\mathbb{C%
}}M\right) \right) ,[M]\right\rangle ,
\end{equation}%
where
\begin{align} \label{Wibundles}
\begin{split}
\Theta (T_{\mathbb{C}}M)&=\overset{\infty }{\underset{m=1}{\otimes }}S_{q^{m}}(\widetilde{T_{\mathbb{C}}M}),\\
\Theta_1(T_{\mathbb{C}}M)&=\bigotimes_{m=1}^\infty\Lambda_{q^m}(\widetilde{T_\CC M}), \\
\Theta_2(T_{\mathbb{C}}M)&=\bigotimes_{m=1}^\infty \Lambda_{-q^{m-{1\over 2}}}(\widetilde{T_{\mathbb C}M})
\end{split}
\end{align}
are the Witten bundle introduced in \cite{W86}.

By the Atiyah-Singer index theorem, when the manifold $M$ is spin,
the elliptic genera can be expressed analytically as index of the twisted Dirac operators
\[
Ell_1(M)= \ind(d_s\otimes(\Theta \left( T_{\mathbb{C%
}}M\right)\otimes\Theta_1 \left( T_{\mathbb{C%
}}M\right))\in \Z[[q]],
\]
\[
Ell_2(M)= \ind(D\otimes (\Theta \left( T_{\mathbb{C%
}}M\right)\otimes\Theta_2 \left( T_{\mathbb{C%
}}M\right))\in \Z[[q^{1/2}]],
\]
and further when $M$ is spin, Witten genus can be expressed analytically as index of the twisted Dirac operators,
\[
W(M)=\ind(D\otimes \Theta \left( T_{\mathbb{C%
}}M\right))\in \mathbb{Z}[[q]],
\]
where $d_s$ is the signature operator and $D$ is the Atiyah-Singer spin Dirac operator on $M$ (cf. \cite{HBJ1992}).

Physically, they are partition functions of two-dimensional quantum field theories \cite{W86} and formally the indices of (twisted) Dirac operator in free loop spaces \cite{W87}.

The purely topological definition of elliptic genera were given by Ochanine \cite{Och} and Landweber-Stong \cite{LS}  The theory of elliptic genera and Witten genus lead to the profound theory of elliptic cohomology \cite{Hop} in algebraic topology.

Elliptic genera and Witten genus have the following modularity properties \cite{HBJ1992, Liu1995, Za}.
\begin{lem} \label{mod1} (i) $Ell_1(M)$ and $Ell_2(M)$ are level 2 modular forms of weight $2m$ and they are modularly related as
\be Ell_1(M, -1/\tau)={(2\tau)}^{2m}Ell_2(M, \tau). \ee
(ii) $W(M)$ is a modular form of weight $2m$, when $M$ is string (i.e. spin and the spin class $\frac{p_1(M)}{2}=0$).
\end{lem}

Let
\be \Theta_2 \left( T_{\mathbb{C%
}}M\right)=B_0(T_\CC M)+B_1(T_\CC
M)q^{1\over2}+\cdots, \ee where the $B_i$'s,
are elements in the semi-group formally generated by Hermitian
vector bundles over $M$. Then
\be  Ell_2(M)=\ind (D\otimes B_0(T_\CC M))+\ind(D\otimes B_1(T_\CC
M))q^{1/2}+\cdots. \ee

\begin{lem} \label{mo}
$Ell_2(M)$ is determined by $ \ind (D\otimes B_k(T_\CC M)), 0 \leq k \leq \frac{m}{2}$, where $B_k(T_\CC M)$ involves linear combinations of tensor product of $TM$
at most to power ${[\frac{m}{2}]}$.
\end{lem}
\begin{proof} See \cite{Liu1995}.
\end{proof}

Let
\be \Theta \left( T_{\mathbb{C%
}}M\right)=W_0(T_\CC M)+W_1(T_\CC
M)q+\cdots, \ee where the $W_i$'s,
are elements in the semi-group formally generated by Hermitian
vector bundles over $M$. Then
\be  W(M)=\ind (D\otimes W_0(T_\CC M))+\ind(D\otimes W_1(T_\CC
M))q+\cdots. \ee
\begin{lem} \label{mod0}
Let $(M, g)$ be string $4m$ dimensional Riemannian manifold, then
$W(M)$ is determined by finite many (depending only on $m$) $\ind(D\otimes W_j(T_\CC M))$.
\end{lem}
\begin{proof} See \cite{HBJ1992}. Actually let
\be
\begin{split}
&E_{4}(\tau)=1+240(q+9q^{2}+28q^{3}+\cdots),\\
&E_{6}(\tau)=1-504(q+33q^{2}+244q^{3}+\cdots)\\
\end{split}
\ee
be the Eisenstein series. Their weights are $4, 6$ respectively, algebraically independent and freely generate $\mathcal{M}_\R(SL(2, \Z)).$

Let
\be
\Theta \left( T_{\mathbb{C}}M\right)=W_0(T_\CC M)+W_1(T_\CC M)q+\cdots,
\ee
where the $W_i$'s, are elements in the semi-group formally generated by Hermitian vector bundles over $M$. Then
\be
W(M)=\ind (D\otimes W_0(T_\CC M))+\ind(D\otimes W_1(T_\CC M))q+\cdots.
\ee
Comparing $q$-coefficients, we see
\[
W(M)=\sum_{4i+6j=2m, \ i, j\geq 0} a_{ij}E_{4}(\tau)^iE_{6}(\tau)^j,
\]
where $a_{i j}$ is determined by finite many (depending only on $m$) $\ind(D\otimes W_j(T_\CC M))$.
\end{proof}

One of most striking properties about elliptic genus is rigidity. Let $M$ be a closed smooth manifold and $P$ be a Fredholm operator on $M$. We assume that a compact connected Lie group $G$ acts on $M$ nontrivially and that $P$ is $G$-equivariant, by which we mean it commutes with the $G$ action. Then the kernel and cokernel of $P$ are finite dimensional representations of $G$. The equivariant index of $P$ is the virtual character of $G$ defined by
\begin{equation}
\ind(h, P)=\tr\big[h\big|_{\ker P}\big]-\tr\big[h\big|_{\cok P}\big],\quad h\in G.
\end{equation}
$P$ is said to be \emph{rigid} for this $G$ action if $\ind(h,P)$ does not
depend on $h\in G$.
Motivated by physics, Witten conjectured that the operators $d_s\otimes(\Theta \left( T_{\mathbb{C%
}}M\right)\otimes\Theta_1 \left( T_{\mathbb{C%
}}M\right)$ and   $D\otimes (\Theta \left( T_{\mathbb{C%
}}M\right)\otimes\Theta_2 \left( T_{\mathbb{C%
}}M\right)$ should be rigid. The Witten conjecture was first proved by Taubes \cite{Tau1989} and Bott-Taubes \cite{BT1989}.
In \cite{Liu1995, Liu1996}, using the modular invariance property, Liu presented a simple and unified proof as well as various generalizations of the Witten conjecture.
\par

A celebrated result asserts that a string manifold with a nontrivial $S^3$-action has vanishing Witten genus, see \cite{Liu1995} and \cite{Des1994}.
On the other hand, concerning the application to positive Ricci curvature, Stolz conjectures that if $M$ is a smooth closed string manifold of dimension $4m$
admitting a Riemannian metric with positive Ricci curvature, then the Witten genus $W(M)$ vanishes \cite{Sto1996}. So far the conjecture is still open.

\bibliographystyle{alpha}
\bibliography{CGHref}
\end{document}